\documentclass[a4paper, dvipdfmx, 12pt]{amsproc}

\usepackage{fullpage}
\usepackage{amsmath, amsthm, amssymb, mathtools, mathrsfs}

\usepackage{mathtools}

\usepackage{hyperref}

\usepackage{chessfss}
\usepackage{tikz}
\usetikzlibrary{intersections, calc, arrows.meta}

\usepackage{graphicx}
\usepackage{here}
\usepackage{time}

\usepackage{xcolor}
\usepackage[capitalize,nameinlink,noabbrev,nosort]{cleveref}
\hypersetup{
	colorlinks=true,       % false: boxed links; true: colored links
	linkcolor=blue,          % color of internal links
	citecolor=red,        % color of links to bibliography
	filecolor=brown,      % color of file links
	urlcolor=brown,           % color of external links
}

\makeatletter
\@namedef{subjclassname@2020}{%
  \textup{2020} Mathematics Subject Classification}
\makeatother

\newtheorem{theoremcounter}{Theorem Counter}[section]

\theoremstyle{definition}
\newtheorem{dfn}[theoremcounter]{Definition}
\newtheorem{remark}[theoremcounter]{Remark}

\theoremstyle{plain}
\newtheorem{lem}[theoremcounter]{Lemma}

\newtheorem{thm}[theoremcounter]{Theorem}

\numberwithin{equation}{section}

\newcommand{\smat}[1]{\bigl(\begin{smallmatrix}#1\end{smallmatrix}\bigr)}

\begin{document}
% --------------------------------------------------------------------------

\title[]{The $p$-adic constant for mock modular forms associated to CM forms. } 

\author{Ryota Tajima} 
\address{Kyushu university}
\email{ryota.tajima.123@gmail.com}

\thanks{The author was supported by JSPS KAKENHI Grant Number JP23KJ1720 and WISE program (MEXT) at Kyushu University. %\framebox{\today \ \now}
}

%05A19: Combinatorial identities, bijective combinatorics
%11B68: Bernoulli and Euler numbers and polynomials

% --------------------------------------------------------------------------

\maketitle

% --------------------------------------------------------------------------
\begin{abstract}
Let $g \in S_{k}(\Gamma_{0}(N))$ be a normalized newform and $f$ be a harmonic Maass form that is good for $g$. The holomorphic part of $f$ is called a mock modular form and denoted by $f^{+}$.
For odd prime $p$, K. Bringmann, P. Guerzhoy, and B. Kane obtained a $p$-adic modular form of level $pN$ from $f^{+}$ and a certain $p$-adic constant $\alpha_{g}(f)$ in \cite{bringmann2012mock}. When $g$ has complex multiplication by an imaginary quadratic field $K$ and $p$ is split in $\mathcal{O}_{K}$, it is known that $\alpha_{g}(f)$ is zero. On the other hand, we do not know much about $\alpha_{g}(f)$ for an inert prime $p$. In this paper, we prove that $\alpha_{g}(f)$ is a $p$-adic unit when $p$ is inert in $\mathcal{O}_{K}$ and $\dim_{\mathbb{C}}S_{k}(\Gamma_{0}(N))=1$. \end{abstract}

\section{Introduction}
A mock modular form is the holomorphic part $f^{+}(z)$ of a harmonic Maass form $f(z)$. The non-holomorphic part of $f$ is connected to a cusp form by the differential operator $\xi_{2-k}$ that maps from harmonic Maass forms to cusp forms
\begin{align*}
\xi_{2-k}:=2iy^{2-k}\overline{( \dfrac{\partial }{\partial \overline{z}}) } :H_{2-k} ( \Gamma _{0}( N) ) \rightarrow S_{k}( \Gamma _{0}( N) ).
\end{align*}
The image of $f(z)$ by $\xi_{2-k}$ is called the shadow of $f^{+}$. For a cusp form $g$, we consider lifts of $g$ by $\xi_{2-k}$ since $\xi_{2-k}$ is surjective. In \cite{bruinier2008differential}, J. H. Bruinier, K. Ono, and R. C. Rhoades found lifts that satisfy some algebraic properties and they called them good for $g$. (cf. Definition \ref{good}.) For example, if $g$ is a normalized newform with complex multiplication and $f$ is good for $g$, then all coefficients of $f^{+}$ are algebraic.

It is a fundamental problem to find a direct relation between mock modular forms and shadows. In \cite{bringmann2012mock}, K. Bringmann, P. Guerthoy, and B. Kane revealed the $p$-adic relation between a normalized newform and the holomorphic part of their good lifts. We state this result more precisely.

Let $g \in S_{k}(\Gamma_{0}(N))$ be a normalized newform with complex multiplication by an imaginary quadratic field $K$ and $f \in H_{2-k} ( \Gamma _{0}( N) )$ good for $g$. Let $p$ be a prime number such that $p \nmid N$ and inert in $\mathcal{O}_{K}$. We define two operators $U(p)$, $V(p)$, and $D^{k-1}$ acting a formal power series by 
\begin{align*}
& U(p)\left(\sum_{n \in \mathbb{Z}} C(n)q^{n}\right):=\sum_{n \in \mathbb{Z}} C(pn)q^{n}
\\
& V(p)\left(\sum_{n \in \mathbb{Z}} C(n)q^{n}\right):=\sum_{n \in \mathbb{Z}} C(n/p)q^{n}
\\
& D^{k-1}\left(\sum_{n \in \mathbb{Z}} C(n)q^{n}\right):=\sum_{n \in \mathbb{Z}} n^{k-1}C(n)q^{n}.
\end{align*}
For a $p$-adic number $\gamma$, we define the formal power series $\widetilde{\mathcal{F}}_{\gamma} \in \mathbb{C}_{p}[[q]][q^{-1}]$ by
\begin{align*}
\widetilde{\mathcal{F}}_{\gamma}:=f^{+}-\gamma E_{g|V(p)}=\sum _{n\gg -\infty} \left( C_{f}^{+}(n) -n^{1-k}C_{g}(n/p) \right)q^{n}=\sum _{n\gg -\infty} n^{1-k}d_{\gamma}(n)q^{n}.
\end{align*}

\begin{thm}[{\cite[Proposition 1.4]{bringmann2012mock}}]\label{Thm_saisyo}
Keep the notation above. Then for all but exactly one $\gamma \in\mathbb{C}_{p}$, we have the $p$-adic limit
\begin{align*}
\lim _{m\rightarrow \infty }\dfrac{(D^{k-1}\widetilde{\mathcal{F}}_{\gamma}) | U ( p^{2m+1})}{d_{\gamma}( p^{2m+1}) }=g.
\end{align*}
\end{thm} 
They also showed a theorem similar to Theorem \ref{Thm_saisyo} when $g$ does not have complex multiplication.
We denote the exceptional constant $\gamma$ of Theorem \ref{Thm_saisyo} by $\alpha_{g}(f)$. In \cite{bringmann2012mock}, K. Bringmann, P. Guerzhoy, and B. Kane showed a remarkable result that $\widetilde{\mathcal{F}}_{\alpha_{g}(f)}$ is not only a formal power series but also a $p$-adic modular form. Similarly, if $g$ does not have complex multiplication, they showed that there exists precisely one $\alpha_{g}(f) \in\mathbb{C}_{p}$ such that $\widetilde{\mathcal{F}}_{\alpha_{g}(f)}$ is a $p$-adic modular form. In order to develop the $p$-adic theory of mock modular forms, it is important to investigate the $p$-adic constant $\alpha_{g}(f)$. For example, it is an interesting question whether $\alpha_{g}(f)$ is zero or not. If $\alpha_{g}(f)$ is not zero, we can choose $\gamma=0$ in Theorem \ref{Thm_saisyo}. Therefore we recover the shadow $g$ from only a mock modular form $f^{+}$. 

In this paper, we assume that $g$ has complex multiplication. Then $\alpha_{g}(f)$ is indepedent on the lifts of $g$ and we denote $\alpha_{g}(f)$ by $\alpha_{g}$ from this. If $p$ is split in $\mathcal{O}_{K}$, then $\alpha_{g}=0$ by \cite{bringmann2012mock}. However, it is not known that whether $\alpha_{g}$ is zero or not when $p$ is inert in $\mathcal{O}_{K}$. It is known no example such that $\alpha_{g}=0$ and one example such that $\alpha_{g}\not=0$ when $p$ is inert in $\mathcal{O}_{K}$. 
\begin{thm}[{\cite{guerzhoy2010p}}]\label{example}

Let $g(z) := \eta ( z) ^{8} \in S_{4}( \Gamma _{0}( 9) ) $.
Then $g$ has complex multiplication by $K=\mathbb{Q}(\sqrt{-3})$, and the following statements are hold.

 (1) There exists a good lift $f$ of $g$ such that 
 \begin{align*}
 D^{k-1}(f^{+})=-\eta ( 3z) ^{8} ( \dfrac{\eta ( z) ^{3}}{\eta (9z) ^{3}}+3) ^{2} = \sum C( n) q^{n}
\end{align*}

 (2) Let $p$ be inert in $\mathcal{O}_{K}$.
If $p^{3}\nmid C( p)$, then $\alpha_{g} \not=0$.

\end{thm}

\begin{remark}
It was shown that $p^{3}\nmid C( p)$ for all inert primes $p<32500$ by \cite{guerzhoy2010p}.
In 2022, Hanson and Jameson show that $p^{3}\nmid C(p)$ for all inert primes in \cite{hanson2022cusp}.
\end{remark}

\begin{remark}
The operator $D^{k-1}$ defines the map from harmonic Maass forms to weakly holomorphic modular forms
 \begin{align*}
D^{k-1}:H_{2-k} ( \Gamma _{0}( N) ) \rightarrow M_{k}^{!}( \Gamma _{0}( N) )
\end{align*}
and kills the non-holomorphic part of $f$.
\end{remark}

 In this paper, we show that if $\dim_{\mathbb{C}}S_{k}(\Gamma_{0}(N))=1$ and $p$ is odd prime and inert in $\mathcal{O}_{K}$, then $\alpha_{g} \not= 0$. Furthermore, we also determine the $p$-adic valuation of $\alpha_{g}$. We state our main theorem.

\begin{thm}\label{main}

Suppose that $k$ is an even integer and $N$ is a natural number. Let $g \in S_{k}(\Gamma_{0}(N))$ be a normalized newform with complex multiplication by an imaginary quadratic field $K$ and $p$ an odd prime number.
Assume that $p$ is inert in $\mathcal{O}_{K}$ and $p \nmid N$.
If $\dim_{\mathbb{C}}$ $S_{k}(\Gamma_{0}(N))=1$, then $\alpha_{g}$ is a $p$-adic unit.

\end{thm}

I will explain the idea of proof. In the proof of (2) in Theorem \ref{example}, they used the fact that $D^{k-1}(f^{+})$ is equal to a weakly holomorphic modular form  $F$ defined by the Dedekind's eta-function (cf. (1) in Theorem \ref{example}.) 
However, an explicit calculation of the holomorphic part of good lifts is very difficult in general. This difficulty comes from the image space of $D^{k-1}$ has infinite dimension over $\mathbb{C}$. In this paper, we consider the quotient space of $M_{k}^{!}( \Gamma _{0}( N) )$ and we denote this quotient space by $\widehat{S_{k}}^{\#, 0}( \Gamma _{0}( N) )$. (cf. Lemma \ref{isom}.)  We show that $\widehat{S_{k}}^{\#, 0}( \Gamma _{0}( N) )$ has finite dimension over $\mathbb{C}$ and we can regard $D^{k-1}$ as the map to $\widehat{S_{k}}^{\#, 0}( \Gamma _{0}( N) )$
\begin{align*}
D^{k-1}:H_{2-k} ( \Gamma _{0}( N) ) \rightarrow \widehat{S_{k}}^{\#, 0}( \Gamma _{0}( N) ).
\end{align*} 
It is easily shown that 
\begin{align*}
D^{k-1}(f^{+})=cF \text{ in } \widehat{S_{k}}^{\#, 0}( \Gamma _{0}( N) )
\end{align*}
 for some non-zero scalar $c$. Lastly, we evaluate the error term between $D^{k-1}(f)$ and $cF$.  
From the definition of $\widehat{S_{k}}^{\#, 0}( \Gamma _{0}( N) )$, there exists a weakly holomorphic modular form $h$ and a complex number $d$ such that
\begin{align}\label{intoro siki}
D^{k-1}(f^{+})=cF+D^{k-1}(h)+dg \text{ in } M_{k}^{!}( \Gamma _{0}( N) ).
\end{align}
The error term coming from $g$ is zero since $p$ is inert in $\mathcal{O}_{K}$.
Considering the Galois action on \ref{intoro siki}, it is shown that $h$ is defined over some algebraic field. This fact implies that the error term coming  from $D^{k-1}(h)$ is equal to $0$. Thus we conclude that $\alpha_{g}\not=0$.

\section{Acknowledgements}
The author would like to show my greatest appreciation to Professor Shinichi Kobayashi for giving many numerous and extremely helpful comments on earlier versions of this paper. I also thank Professor Toshiki Matsusaka for giving many comments about mock modular forms. I was supported by JSPS KAKENHI Grant Number JP23KJ1720 and WISE program (MEXT) at Kyushu University.

 \section{Harmonic Maass forms and weakly holomorphic modular forms}
 In this section, we introduce facts for harmonic Maass forms and weakly holomorphic modular forms.
 
Throughout, let $\mathbb{H}$ be the upper-half of the complex plane and $z=x+iy \in \mathbb{H}$ with $x, y \in \mathbb{R}$

\begin{dfn}
Let $k \in \mathbb{Z}$ and $N \in \mathbb{N}$.
Then a harmonic Maass form of weight $k$ on $ \Gamma _{0}( N)$ is any smooth function $f$ on $\mathbb{H}$ satisfying:

(1)$f\left(\dfrac{az+b}{cz+d}\right)=(cz+d)^{k}f(z)$ for all $\gamma =  \smat{a & b \\ c & d} \in \Gamma _{0}( N)$.

(2)$\Delta_{k}f=0$, where
 $\Delta_{k}=-y^{2}\left( \dfrac{\partial ^{2}}{\partial x^{2}}+\dfrac{\partial ^{2}}{\partial y^{2}}\right) +iky\left( \dfrac{\partial }{\partial x}+i\dfrac{\partial }{\partial y}\right) $. 
 
(3)There is a polynomial $P_{\infty}(z) \in \mathbb{C}[q^{-1}]$ such that
\begin{align*}
f(z)-P_{\infty}(z)=\text{O} (e^{-\varepsilon y})  \text{ as } y  \rightarrow \infty \text{ for some } \epsilon >0.
\end{align*}
Analogous conditions are required at all cusps.
 \end{dfn}
 
\vspace{0.4cm}
 
We denote the vector space of these harmonic Maass forms by $H_{k}( \Gamma _{0}( N) )$.

Every harmonic Maass form $f(z)$ of weight $2-k$ has a Fourier expansion of the form
\begin{align*}
f=\sum _{n\gg -\infty}C_{f}^{+}( n) q^{n}+\sum _{n <0}C_{f}^{-}( n) \Gamma ( k-1,4\pi  \left| n\right| y) q^{n}.
\end{align*}

Obviously, each $f(z)$ is the sum of two disjoint pieces, the holomorphic part of $f(z)$

\begin{align*}
f^{+}(z):=\sum _{n\gg -\infty}C_{f}^{+}( n) q^{n},
\end{align*}

and the non-holomorphic part of $f(z)$

\begin{align*}
f^{-}(z):=\sum _{n <0}C_{h}^{-}( n) \Gamma ( k-1,4\pi  \left| n\right| y) q^{n}.
\end{align*}

In addition, $\sum _{n\leq 0}C_{f}^{+}( n) q^{n}$ is called the principal part of $f(z)$ at the cusp $\infty$.

\begin{remark}
Every weakly holomorphic modular form $f(z)$ is in $H_{k}( \Gamma _{0}( N) )$ with $f^{-}(z)=0.$
\end{remark}

\begin{dfn}
A mock modular form is the holomorphic part of a harmonic Maass form.
\end{dfn}
 
 \begin{thm}[{\cite{bruinier2008differential}}]\label{xi}
 Suppose that $k$ is an integer greater than or equal to 2.
  We define two operators $D:=\dfrac{1}{2 \pi i}\dfrac{d}{dz}$ and $\xi _{w}:=2iy^{w}\overline{( \dfrac{\partial }{\partial \overline{z}}) }$ where $w \in \mathbb{Z}$.
  \\
Then
\begin{align*}
& D^{k-1}:H_{2-k} ( \Gamma _{0}( N) ) \rightarrow M_{k}^{!}( \Gamma _{0}( N) ), 
\\
 &\xi_{2-k}:H_{2-k} ( \Gamma _{0}( N) ) \rightarrow S_{k}( \Gamma _{0}( N) ) 
 \end{align*}
 and
 \begin{align*}
 D^{k-1}(f^{-})=0, \xi_{2-k}(f^{+})=0.
 \end{align*}

 In particular, $\xi_{2-k}:H_{2-k} ( \Gamma _{0}( N) ) \rightarrow S_{k}( \Gamma _{0}( N) ) $ is surjective.
 \end{thm}

The image of $f$ by $\xi_{2-k}$ is called the shadow of $f^{+}$.

 \begin{thm}[{\cite[Lemma 2.3]{bringmann2012coefficients}}]\label{thm: pri nonzero}
If $f \in H_{2-k}(\Gamma_{0}(N))$ has the property that $\xi_{2-k}(f)\not=0$, then the principal part of $f$ is nonconstant for at least one cusp.

 \end{thm}

 \begin{dfn}\label{good}
 Let $g \in S_{k}( \Gamma _{0}( N) )$ be a normalized newform and $F_{g}$ be the number field obtained by adjoining the coefficients of $g$ to $\mathbb{Q}$.
 We say that a harmonic Maass form $f \in H_{2-k} ( \Gamma _{0}( N) )$ is good for $g$ if it satisfies the following properties.

 (1)The principal part of $f$ at the $\infty$ belongs to $F_{g}[q^{-1}]$.
 
 (2)The principal part of $f$ at the other cusps of $\Gamma_{0}(N)$ are constant.
 
 (3)We have that $\xi _{2-k}( f) =\dfrac{g}{\left\| g\right\| ^{2}}$.
  \end{dfn}
 
 \begin{thm}[{\cite[Theorem 1.3]{bruinier2008differential}}]\label{cof of mock}
 Let $g \in S_{k}( \Gamma _{0}( N) )$ be a normalized newform with complex multiplication. 
 If $f \in H_{2-k} ( \Gamma _{0}( N) )$ is good for $g$, 
 then there exists a positive integer $M$ such that all coefficients of $f^{+}$ are in $F_{g}(\zeta_{M})$, where $\zeta_{M}:=e^{\frac{2\pi i}{M}}$.
 \end{thm}
 
 \begin{lem}[{\cite[Proposition 2.]{hanson2022cusp}}]\label{list of 0}
 
 The one-dimensional spaces $S_{k}(\Gamma_{0}(N))$ which satisfies the assumption of Theorem \ref{main} are only
\begin{align*}
S_{2}( \Gamma _{0}( 27) ), S_{2}( \Gamma _{0}( 32) ), S_{2}( \Gamma _{0}( 36) ), S_{2}( \Gamma _{0}( 49) ), 
 S_{4}( \Gamma _{0}( 9) ).
 \end{align*}

In addition, every genus of $X_{0}(N)$ is $0$ or $1$ where $N=27, 32, 36, 49, 9$.
 
 \end{lem}

 \begin{dfn}
 We define two subspaces of weakly holomorphic modular forms
 \begin{align*}
&M_{k}^{\# }( \Gamma _{0}( N) ) :=\{ f\in M_{k}^{!}( \Gamma _{0}( N) )  \mid  \text{$f$ is holomorphic at every cusp except possibly $\infty$}\},
\\
&S_{k}^{\# ,0}( \Gamma _{0}( N) ) :=\{ f\in M_{k}^{\#}( \Gamma _{0}( N) )  \mid \text{$f$ vanishes at every cusp except possibly $\infty$}\}.
 \end{align*}
 \end{dfn}
 
 \begin{lem}[{\cite[Theorem 1.3.]{hwang2021arithmetic}}] \label{isom}
 Let k be a positive even integer and N be a positive integer for which the
genus of $\Gamma_{0}(N)$ is zero or one.
We define the space $\widehat{S_{k}}^{\# ,0} ( \Gamma _{0}( N) )$ by
\begin{align*}
\widehat{S_{k}}^{\# ,0} ( \Gamma _{0}( N) ) :=\dfrac{S_{k}^{\# ,0}( \Gamma_{0}( N) ) }{D^{k-1}( M_{2-k}^{\# }( \Gamma _{0}( N) )) \oplus S_{k}( \Gamma_{0}( N) ) }.
 \end{align*}
 Then
 \begin{align*}
\dim \widehat{S_{k}}^{\# ,0} ( \Gamma _{0}( N) ) = \dim S_{k}( \Gamma _{0}( N) ).
 \end{align*}
 
 \end{lem}

\begin{lem}[{\cite[Theorem 1.]{hanson2022cusp}}]\label{hanson}
Suppose that $S_{k}( \Gamma _{0}( N) )$ is one-dimensional and that the unique normalized
cusp form $g$ has complex multiplication by $K$.
\\
There exists
\begin{align*}
F=-q^{-1}+\sum ^{\infty }_{n=2}C_{F}( n) q^{n} \in S_{k}^{\# ,0}( \Gamma _{0}( N) ) \cap \mathbb{Z}[[q]][q^{-1}]
\end{align*}
such that for every odd prime $p$ which is inert in $\mathcal{O}_{K}$ and every integer $m \geq 0$ we have that
\begin{align*}
v_{p}( C_{F}( p^{2m+1}) )=(k-1)m.
\end{align*}
\end{lem}

\begin{proof}
The above result except for $F \in S_{k}^{\# ,0}( \Gamma _{0}( N) )$ is clear by \cite[Theorem 1.]{hanson2022cusp}.
The function $F$ is defined by $F_{1}$ in \cite[Proposition3]{hanson2022cusp}.
It is clear that  $F_{1} \in S_{k}^{\# ,0}( \Gamma _{0}( N) )$ from the definition of $F_{1}$.
\end{proof}

\begin{remark}
When $(k,N)=(4,9)$, we have $F=-\eta ( 3z) ^{8} ( \dfrac{\eta ( z) ^{3}}{\eta ( 9z) ^{3}}+3) ^{2}$. (See Theorem \ref{example})
\end{remark}

\begin{lem}
Let $g$ be a normalized newform in $S_{k}( \Gamma _{0}( N) )$.

If $f \in H_{2-k} ( \Gamma _{0}( N) )$ is good for $g$, 
then 
$D^{k-1}(f) \in S_{k}^{\# ,0}( \Gamma _{0}( N) )$.
\end{lem}

\begin{proof}
 We have $D^{k-1}(f) \in M_{k}^{!}( \Gamma _{0}( N) )$ by \cite[Theorem 1.2]{bruinier2008differential}.
 We will show that the constant term of $D^{k-1}(f)$ is zero at every cusp of $\Gamma _{0}( N)$.
Let $s$ be a cusp of $\Gamma _{0}( N)$, and $h$ be the width of $s$.
\\
We denote the Fourier expansion of $f^{+}$ at $s$ by $\sum _{n\gg -\infty }C_{s}( n) q_{h}^{n}$ where $q_{h}:= e^{\frac{2\pi i}{h}}$.
Then the Fourier expansion of $D^{k-1}(f)$ at $s$ is 
\begin{align*}
D^{k-1}(f)=D^{k-1}(f^{+})=\sum _{n\gg -\infty }(\dfrac{n}{h})^{k-1}C_{s}( n) q_{h}^{n}.
\end{align*}

Therefore the constant term of $D^{k-1}(f)$ is zero at every cusp of $\Gamma _{0}( N)$.
\\
We will show that $D^{k-1}(f)$ is holomorphic at every cusp exept for $\infty$.
Let $s \in \mathbb{Q}$ be a cusp of $\Gamma _{0}( N)$ and $h$ be the width of $s$.
We denote the Fourier expansion of $f^{+}$ at $s$ by $\sum _{n\gg -\infty }C_{s}( n) q_{h}^{n}$.
Since $f$ is good for $g$, $C_{s}( n)=0$ holds for all $n <0$.
Hence  $D^{k-1}(f)$ is holomorphic at every cusp exept for $\infty$.
\end{proof}

\section{$p$-adic properties of mock modular forms}

In this section, we recall  $p$-adic properties of mock modular forms.

 From now on, we fix an algebraic closure $\overline{\mathbb{Q}}_{p}$ along with embedding $\iota \colon \overline{\mathbb{Q}}\rightarrow \overline{\mathbb{Q}}_{p}$ for each prime number $p$. We denote the $p$-adic closure by $\mathbb{C}_{p}$ and normalize the $p$-adic valuation so that $v_{p}(p)=1$.
 
 Let $g \in S_{k}(\Gamma_{0}(N))$ be a normalized newform with complex multiplication by $K$ and $f \in H_{2-k}(\Gamma_{0}(N))$ be good for $g$. We denote the holomorphic part of $f$ by $f^{+}$. We define the Eichler integral of $g$ by
\begin{align*}
E_{g}( z) :=\sum _{n>0}n^{1-k}C_{g}( n) q^{n}
\end{align*}
where $C_{g}(n)$ denotes the $n$-th coefficient of $g$.
For $\gamma \in \mathbb{C}_{p}$, we define
\begin{align*}
\widetilde{\mathcal{F}}_{\gamma}:=f^{+}-\gamma E_{g|V(p)}.
\end{align*}
Let $\beta, \beta'$ be the roots of the polynomial $X^{2}-C_{g}(p)X+p^{k-1}$ such that $v_{p}( \beta ) \leq v_{p}( \beta ') $. 

\begin{lem}[{\cite[Proposition 2.3]{guerzhoy2010p}}]
Let $g$ be a normalized newform with complex multiplication by $K$ and $p$ is inert in $\mathcal{O}_{K}$. Then 
\begin{align*}
\displaystyle\lim _{m\rightarrow \infty }\dfrac{C_{D^{k-1}(f)}( p^{2m+1}) }{\beta ^{2m}}
\end{align*}
is convergence.
\end{lem}

\begin{thm}[{\cite[Theorem 1.3]{bringmann2012mock}}]
Assume that $p \nmid N$ and  $p$ is inert in $\mathcal{O}_{K}$.
Then there exists exactly one $\alpha_{g} \in \mathbb{C}_{p}$ such that $\widetilde{\mathcal{F}}_{\alpha_{g}}$ is a $p$-adic modular form of weight $2-k$ and level $pN$, given by the $p$-adic limit 
\begin{align*}
\alpha_{g}=\displaystyle\lim _{m\rightarrow \infty }\dfrac{C_{D^{k-1}(f)}(p^{2m+1}) }{\beta ^{2m}}.
\end{align*}
\end{thm}

We will show that $\alpha_{g}$ is well-defined.

\begin{lem}[{\cite[Proposition 2.1]{guerzhoy2010p}}]\label{p-adic valuation}
Let $h \in M_{2-k}^{!}( \Gamma _{0}( N) )$ be defined over some algebraic field. Then there is a real number $A$ such that
\begin{align*}
 v_{p}(C_{D^{k-1}(h)}(p^{2m+1})) \geq (2m+1)(k-1)-A  \text{ for all } m \in \mathbb{N}.
\end{align*}
\end{lem}

If $f$ and $f'$ are good for $g$, then $\xi_{2-k}(f-f')=0$. Therefore $f-f'$ is an element of $M_{2-k}^{!}( \Gamma _{0}(N))$ and defined over $F_{g}(\zeta_{M})$ by Theorem \ref{cof of mock}. By Lemma \ref{p-adic valuation}, we have
\begin{align*}
\displaystyle\lim _{m\rightarrow \infty }\dfrac{\left(C_{D^{k-1}(f)}(p^{2m+1})-C_{D^{k-1}(f')}(p^{2m+1})\right) }{\beta ^{2m}}=0.
\end{align*}
Therefore $\alpha_{g}$ is well-defined.

\section{Proof of the Main Theorem}
In this section, we prove the main theorem.
Firstly, we will show that $\alpha_{g}\not=0$.
\begin{lem}\label{kihonsiki}
Suppose that $S_{k}( \Gamma _{0}( N) )$ is one-dimensional and that the unique normalized
cusp form $g$ has complex multiplication by $K$. Then there exist $c, d \in \mathbb{C}$ and $h \in M_{2-k}^{\# }( \Gamma _{0}( N) )$ such that 
\begin{equation}\label{kihon eq}
D^{k-1}(f)=cF+D^{k-1}(h)+dg.
\end{equation}
\end{lem}

\begin{proof}
From Lemma \ref{list of 0} and Lemma \ref{isom}, we obtain that $\dim$ $\widehat{S_{k}}^{\# ,0} ( \Gamma _{0}( N) )=1$.
\\
Therefore we can show this lemma by  $D^{k-1}(f), F \in \widehat{S_{k}}^{\# ,0} ( \Gamma _{0}( N) )$.
\end{proof}

\begin{lem}\label{c not 0}
The constant $c$ in Lemma \ref{kihonsiki} is not zero.
\end{lem}

\begin{proof}
Assume that $c=0$.
We write
$f^{+}=\sum _{n\gg -\infty}C_{f}^{+}( n) q^{n}$ and 
$h=\sum _{n\gg -\infty}C_{h}( n) q^{n}$.
\\
Since $c=0$, we have
\begin{align*}
D^{k-1}(f)=D^{k-1}(h)+dg.
\end{align*}
Therefore if $n\leq -1$, then
\begin{align*}
C_{f}^{+}(n)=C_{h}(n).
\end{align*}
We put $H =f-h \in H_{2-k} ( \Gamma _{0}( N) )$. Then the principal part of $H$ at $\infty$ is constant.

Let $s \in \mathbb{Q}$ be a cusp of $\Gamma_{0}(N)$.
Since $f$ is good for $g$, the principal part of $f$ at $s$ is constant.
Therefore the principal part of $H$ at $s$ is constant. 
\\
Consequently, the principal part of $H$ is constant at all cusps and
\begin{align*}
\xi_{2-k}(H)=\xi_{2-k}(f)=\dfrac{g}{\left\| g\right\| ^{2}}\not=0.
\end{align*}

This contradicts Theorem \ref{thm: pri nonzero}.
\end{proof}

\begin{lem}\label{c in F_g}
There exists a positive integer $M$ such that the constant $c, d$ in Lemma \ref{kihonsiki} are in  $F_{g}(\zeta_{M})$.
\end{lem}

\begin{proof}
In a manner similar to the proof of Lemma \ref{c not 0}, we can show that $F\not=0$ as an element of $\widehat{S_{k}}^{\# ,0} ( \Gamma _{0}( N) )$.
By Theorem \ref{cof of mock}, we have 
\begin{align*}
D^{k-1}(f)=c^{\sigma}F+D^{k-1}(h^{\sigma})+d^{\sigma}g 
\end{align*}
for all $\sigma \in \text{Aut}(\mathbb{C} / F_{g}(\zeta_{M}))$.
From the discussion of \cite[Theorem1.3]{bruinier2008differential}, $h^{\sigma} \in M_{2-k}^{\# }( \Gamma _{0}( N) )$ holds. Therefore as an element of $\widehat{S_{k}}^{\# ,0} ( \Gamma _{0}( N) )$, we have
\begin{align*}
D^{k-1}(f)=cF=c^{\sigma}F.
\end{align*}
Since $F\not=0$, we obtain $c \in F_{g}(\zeta_{M})$.
Therefore, we have
\begin{align*}
D^{k-1}(f)=cF+D^{k-1}(h^{\sigma})+d^{\sigma}g 
\end{align*}
for all $\sigma \in \text{Aut}(\mathbb{C} / F_{g}(\zeta_{M}))$.
Therefore
\begin{align*}
(d-d^{\sigma})g=D^{k-1}(h^{\sigma}-h)
\end{align*}
holds.
Since $F$ is written by $F=-q^{-1}+\sum ^{\infty }_{n=2}C( n) q^{n}$ and \eqref{kihon eq}, 
\begin{align*}
C_{h}(-1)=C_{f}^{+}(-1)+1
\end{align*}
and
\begin{align*}
C_{h}(n)=C_{f}^{+}(n)
\end{align*}
for all $n \leq -2$.
Consequently, all coefficients of $h$ at negative integers are in $F_{g}(\zeta_{M})$.
Therefore, we have
\begin{align*}
h^{\sigma}-h \in M_{2-k}(\Gamma_0({N})).
\end{align*}
By  Lemma \ref{list of 0} and the assumption of Lemma \ref{kihonsiki}, we have $k=2, 4$.
\\
(1)In case that $k=2$, $D^{k-1}(h^{\sigma}-h)=0$ by $M_{0}(\Gamma_0({N}))=\mathbb{C}$.
\\
(2)In case that $k=4$, $D^{k-1}(h^{\sigma}-h)=0$ by  $M_{-2}(\Gamma_0({N}))=0$.

Consequently, we obtain $d=d^{\sigma}$.
\end{proof}

\begin{lem}
All coefficients of $h$ in Lemma \ref{c in F_g} are element of $ F_{g}(\zeta_{M})$.
\end{lem}

\begin{proof}
By the proof of Theorem \ref{c in F_g}, we have
\begin{align*}
h^{\sigma}-h \in M_{2-k}(\Gamma_0({N})).
\end{align*}
for all $\sigma \in \text{Aut}(\mathbb{C} / F_{g}(\zeta_{M}))$.
If $k=4$, then
\begin{align*}
h^{\sigma}=h
\end{align*}
holds.
Therefore we assume that $k=2$.
Let $C_{h}(0)$ be the $0$-th coefficient of $h$.
Then we have
\begin{align*}
h-C_{h}(0) \in M_{0}^{!}(\Gamma_0({N}))
\end{align*}
and $h-C_{h}(0)$ satisfies \ref{kihon eq}.
Therefore we can assume that the constant term of $h$ is zero.
Thanks to this assumption, 
\begin{align*}
h^{\sigma}=h
\end{align*}
holds.
\end{proof}

\begin{thm}\label{main not 0}
Suppose that $k$ is an even integer and $N$ is a natural number. Let $g \in S_{k}(\Gamma_{0}(N))$ be a normalized newform with complex multiplication by an imaginary quadratic field $K$ and $p$ an odd prime number.
Assume that $p$ is inert in $\mathcal{O}_{K}$ and $p \nmid N$.
If $\dim S_{k}(\Gamma_{0}(N))=1$, then $\alpha_{g}\not=0$.

\end{thm}

\begin{proof}
Since $p$ is inert in $\mathcal{O}_{K}$, we have
\begin{align*}
C_{g}(p^{2m+1})=0 
\end{align*}
$\text{ for all } m \in \mathbb{N}.$
Hence we obtain
\begin{align*}
C_{D^{k-1}(f)}^{+}(p^{2m+1})=cC_{F}(p^{2m+1})+C_{D^{k-1}(h)}(p^{2m+1}).
\end{align*}

Therefore
\begin{align*}
\dfrac{C_{D^{k-1}(f)}(p^{2m+1})}{\beta^{2m}}=c \dfrac{C_{F}(p^{2m+1})}{\beta^{2m}}+\dfrac{c_{D^{k-1}(h)}(p^{2m+1})}{\beta^{2m}}
\end{align*} 
holds.
Since $\beta$ is the root of the polynomial $X^{2}-c_{g}(p)X+p^{k-1}=X^{2}+p^{k-1}$ and Lemma \ref{p-adic valuation}, we have
\begin{align*}
\alpha_{g} = \displaystyle\lim _{m\rightarrow \infty }\dfrac{C_{D^{k-1}(f)}( p^{2m+1}) }{\beta ^{2m}} = c \displaystyle\lim _{m\rightarrow \infty }\dfrac{C_{F}( p^{2m+1}) }{\beta ^{2m}}.
\end{align*}
From Lemma \ref{hanson} and the fact that $ \left\{ x \in \mathbb{C}_{p} \mid v_{p}(x)=0 \right\}$ is closed,
 \begin{align*}
\displaystyle\lim _{m\rightarrow \infty }\dfrac{C_{F}( p^{2m+1}) }{\beta ^{2m}} \not=0
\end{align*}
holds.
By Lemma \ref{c not 0}, we obtain
\begin{align*}
\alpha_{g} \not= 0.
\end{align*}
\end{proof}

Lastly, we will show that $v_{p}(\alpha_{g})=0$.

\begin{lem}[{\cite[Proposition 5.11.]{bringmann2017harmonic}}]\label{Brunier}
Let $f \in H_{2-k}( \Gamma _{0}( N) )$ and $g \in S_{k}( \Gamma _{0}( N) )$.
We denote the principal part of $f$ at cusp $s$ by 
\begin{align*}
\sum_{n<0}C_{f, s}^{+}(n)q_{h_{s}}^{n}
\end{align*} 
and the Fourier expansion of $g$ at cusp $s$ by
\begin{align*}
\sum_{n>0}C_{g, s}(n)q_{h_{s}}^{n}
\end{align*} 
where $h_{s}$ is the width of $s$.
 Then we have
 \begin{align*}
 (\xi_{2-k}(f), g)=\sum_{s : \text{cusp}}\sum_{n<0} C_{f, s}^{+}(n)C_{g, s}(n)
 \end{align*}
 where $(\cdot , \cdot)$ is the Petersson inner product.
\end{lem}
We generalize the pairing that is defined by P. Guerzhoy. (cf. \cite{guerzhoy2008hecke})
\begin{lem}\label{pairing}
We define the pairing
\begin{align*}
\langle \cdot , \cdot \rangle : \widehat{S_{k}}^{\#, 0}( \Gamma _{0}( N) ) \times S_{k} ( \Gamma _{0}( N) ) \rightarrow \mathbb{C} \quad \text{by}
\end{align*}
\begin{align*}
\langle \sum_{n\gg -\infty} a_{n}q^{n}, \sum_{n>0} b_{n}q^{n}\rangle:=\displaystyle\sum_{n < 0} \dfrac{a_{n}b_{-n}}{n^{k-1}}.
\end{align*}
Then this pairing is well-defined.
\end{lem}

\begin{proof}
It is sufficient to show that $\langle D^{k-1}(h), g\rangle=0$ for all $h \in M_{2-k}^{\#}( \Gamma _{0}( N) )$ and $g \in S_{k} ( \Gamma _{0}( N) )$. Since $M_{2-k}^{\#}( \Gamma _{0}( N) ) \subset H_{2-k}( \Gamma _{0}( N) )$ and Lemma \ref{Brunier}, we have
\begin{align*}
\langle D^{k-1}(h), g\rangle=\sum_{n<0}C_{h}(n)C_{g}(n)= (\xi_{2-k}(h), g)=(0, g)=0.
\end{align*}
\end{proof}

\begin{lem}
Let $g$ be a normalized newform. If $f$ is good for $g$ then we have
\begin{align*}
\langle D^{k-1}(f), g\rangle=1.
\end{align*}
\end{lem}

\begin{proof}
This Lemma is followed by Lemma \ref{Brunier} and the definition of ``good for $g$''.
\end{proof}

\begin{thm}
Suppose that $k$ is an even integer and $N$ is a natural number. Let $g \in S_{k}(\Gamma_{0}(N))$ be a normalized newform with complex multiplication by an imaginary quadratic field $K$ and $p$ an odd prime number.
Assume that $p$ is inert in $\mathcal{O}_{K}$ and $p \nmid N$.
If $\dim S_{k}(\Gamma_{0}(N))=1$, then $v_{p}(\alpha_{g})=0$.
\end{thm}

\begin{proof}
From the discussion in Theorem \ref{main not 0}, it is sufficient to show that $c=1$.
From the definition of $\langle \cdot , \cdot \rangle$, we have
\begin{align*}
\langle F , g\rangle=1.
\end{align*}
Since $D^{k-1}(f)=cF \text{ in }  \widehat{S_{k}}^{\#, 0}( \Gamma _{0}( N) )$,
\begin{align*}
1=\langle D^{k-1}(f), g\rangle=\langle cF, g \rangle=c\langle F, g\rangle=c.
\end{align*}
\end{proof}

%\section*{Acknowledgements}

\bibliographystyle{amsplain}%{amsalpha}
\bibliography{References}

\newcommand{\noopsort}[1]{}
\providecommand{\bysame}{\leavevmode\hbox to3em{\hrulefill}\thinspace}
\providecommand{\MR}{\relax\ifhmode\unskip\space\fi MR }
\providecommand{\MRhref}[2]{%
  \href{http://www.ams.org/mathscinet-getitem?mr=#1}{#2}
}
\providecommand{\href}[2]{#2}
\begin{thebibliography}{1}

\bibitem{bringmann2017harmonic}
Kathrin Bringmann, Amanda Folsom, Ken Ono, and Larry Rolen, \emph{Harmonic
  maass forms and mock modular forms: theory and applications}, vol.~64,
  American Mathematical Soc., 2017.

\bibitem{bringmann2012mock}
Kathrin Bringmann, Pavel Guerzhoy, and Ben Kane, \emph{Mock modular forms as
  $p$-adic modular forms}, Transactions of the American Mathematical Society
  \textbf{364} (2012), no.~5, 2393--2410.

\bibitem{bringmann2012coefficients}
Kathrin Bringmann and Ken Ono, \emph{Coefficients of harmonic maass forms},
  Partitions, q-series, and modular forms, Springer, 2012, pp.~23--38.

\bibitem{bruinier2008differential}
Jan~H Bruinier, Ken Ono, and Robert~C Rhoades, \emph{Differential operators for
  harmonic weak maass forms and the vanishing of hecke eigenvalues},
  Mathematische Annalen \textbf{342} (2008), no.~3, 673--693.

\bibitem{guerzhoy2008hecke}
P~Guerzhoy, \emph{Hecke operators for weakly holomorphic modular forms and
  supersingular congruences}, Proceedings of the American Mathematical Society
  \textbf{136} (2008), no.~9, 3051--3059.

\bibitem{guerzhoy2010p}
Pavel Guerzhoy, Zachary~A Kent, and Ken Ono, \emph{p-adic coupling of mock
  modular forms and shadows}, Proceedings of the National Academy of Sciences
  \textbf{107} (2010), no.~14, 6169--6174.

\bibitem{hanson2022cusp}
Michael Hanson and Marie Jameson, \emph{Cusp forms as p-adic limits}, Journal
  of Number Theory \textbf{234} (2022), 349--362.

\bibitem{hwang2021arithmetic}
Jihyun Hwang and Chang~Heon Kim, \emph{Arithmetic of weakly holomorphic hecke
  eigenforms}, Advances in Mathematics \textbf{384} (2021), 107750.

\end{thebibliography}

% --------------------------------------------------------------------------
\end{document}